\documentclass[10pt]{amsart}
\usepackage[margin=1in]{geometry}
\usepackage{
 amsmath,
 amsxtra,
 amsthm,
 amssymb,
 etex,
 hyperref,
 mathrsfs,
 mathtools,
 tikz-cd,
 bbm,
 xr,
 comment}
\usepackage[normalem]{ulem}
\usepackage[toc]{appendix}
\usepackage[all]{xy}
\usepackage{hyperref}
\usepackage{url}
\usepackage{tipa}
\usepackage{dsfont}
\usepackage{ textcomp }
\usepackage{stmaryrd}
\usepackage{amssymb}
\usepackage{mathabx}
\usepackage{amsmath}
\usepackage{amscd}
\usepackage{amsbsy}
\usepackage{comment, enumerate}
\usepackage{color}
\usepackage{mathtools,caption}
\usepackage{tikz-cd}
\usepackage{longtable}
\usepackage[utf8]{inputenc}
\usepackage[OT2,T1]{fontenc}
\usepackage{changepage}
\usepackage{commath}

\DeclareMathOperator{\glob}{glob}

\DeclareMathOperator{\Hom}{Hom}

\DeclareMathOperator{\Gal}{Gal}
\DeclareMathOperator{\rank}{rank}

\DeclareMathOperator{\corank}{corank}
\DeclareMathOperator{\Sel}{Sel}

\DeclareMathOperator{\K}{\mathcal{K}}

\DeclareMathOperator{\cyc}{cyc}

\DeclareMathOperator{\Iw}{Iw}

\newtheorem{theorem}{Theorem}[section]
\newtheorem*{theorem*}{Theorem}
\newtheorem{lemma}[theorem]{Lemma}

\newtheorem{proposition}[theorem]{Proposition}

\newtheorem{defn}[theorem]{Definition}

\numberwithin{equation}{section}
\newtheorem{lthm}{Theorem} % theorems with letters (for intro)

\theoremstyle{remark}
\newtheorem{remark}[theorem]{Remark}
\newtheorem{example}[theorem]{Example}

\newcommand\EatDot[1]{}

\newcommand{\cF}{{\mathcal{F}}}

\newcommand{\cM}{{\mathcal{M}}}
\newcommand{\Image}{\mathrm{Image}}

\newcommand{\CC}{\mathbb{C}}

\newcommand{\QQ}{\mathbb{Q}}
\newcommand{\ZZ}{\mathbb{Z}}
\newcommand{\Qp}{\mathbb{Q}_p}
\newcommand{\Zp}{\mathbb{Z}_p}

\newcommand{\lb}{\llbracket}
\newcommand{\rb}{\rrbracket}

\newcommand{\absolute}[1]{\left\lvert#1\right\rvert}
\definecolor{Green}{rgb}{0.0, 0.5, 0.0}

\newcommand{\cO}{\mathcal{O}}
\newcommand{\Qpn}{L_{n,p}}
\newcommand{\Qcyc}{\QQ_{\cyc}}
\newcommand{\Q}{\mathbb{Q}}
\newcommand{\F}{\mathbb{F}}
\newcommand{\Z}{\mathbb{Z}}
\newcommand{\p}{\mathfrak{p}}
\newcommand{\Char}{\mathrm{Char}}

\newcommand{\of}{{\overline{f}}}
\title[Fine Selmer groups, Control Theorems, and Duality]{Control theorems for fine Selmer groups, and  duality of fine Selmer groups attached to modular forms}

\makeatletter
\let\@wraptoccontribs\wraptoccontribs
\makeatother

\author[J.~Hatley]{Jeffrey Hatley}
\address[Hatley]{
Department of Mathematics\\
Union College\\
Bailey Hall 202\\
Schenectady, NY 12308\\
USA}
\email{hatleyj@union.edu}
\author[D.~Kundu]{Debanjana Kundu}
\address[Kundu]{Pacific Institute of Mathematical Sciences \\ University of British Columbia\\
4176-2207 Main Mall\\
Vancouver, BC, Canada V6T 1Z4}
\email{dkundu@math.ubc.ca}

\author[A.~Lei]{Antonio Lei}
\address[Lei]{D\'epartement de Math\'ematiques et de Statistique\\
Universit\'e Laval, Pavillion Alexandre-Vachon\\
1045 Avenue de la M\'edecine\\
Qu\'ebec, QC\\
Canada G1V 0A6}
\email{antonio.lei@mat.ulaval.ca}

\author[J.~Ray]{Jishnu Ray}
\address[Ray]{School of Mathematics, Tata Institute of Fundamental Research, Dr Homi Bhabha Road, Navy Nagar, Colaba, Mumbai, Maharashtra 400005}
\email{jishnu.ray@tifr.res.in; jishnuray1992@gmail.com}


\begin{document}
\begin{abstract}Let $\mathcal{O}$ be the ring of integers of a finite extension of $\mathbb{Q}_p$.
We prove two control theorems for fine Selmer groups of general cofinitely generated modules over $\mathcal{O}$.
We apply these control theorems to compare the fine Selmer group attached to a modular form $f$ over the cyclotomic $\mathbb{Z}_p$-extension of $\mathbb{Q}$ to its counterpart attached to the conjugate modular form $\overline{f}$.
\end{abstract}

\maketitle

\section{Introduction}
\label{S: Intro}
Let $p$ be a fixed odd prime number.
Let $f$ be a normalized eigen-cuspform of level $N$ and weight $k\ge 2$.
We assume throughout that $p\nmid N$.
We write $\of$ for its conjugate modular form,
that is, the modular form whose Fourier coefficients are given by the complex conjugation of those of $f$.
We fix embeddings $\iota_\infty:\overline{\QQ}\hookrightarrow \CC$ and $\iota_p:\overline{\QQ}\hookrightarrow \CC_p$, which allow us to regard the Fourier coefficients of $f$ and $\of$ as elements of $\CC_p$.
Let $K/\Qp$ be a fixed finite extension that contains all these Fourier coefficients and write $\cO$ for its ring of integers.
Throughout,  $\varpi$ is a fixed uniformizer  of $\cO$.

For $g=f$ or $\of$, let $V_g$ denote the $K$-adic $G_\QQ$-representation attached to $g$ by Deligne \cite{deligne69}.
Our normalization is such that these representations have Hodge--Tate weights $0$ and $1-k$ at $p$, with the convention that the $p$-adic cyclotomic character has Hodge--Tate weight $1$.
Recall that the $K$-linear dual of $V_f$ is
\[
V_f^*:=\Hom_K(V_f,K)\cong V_\of(k-1),
\]
where $M(j)$ denotes the $j$-th Tate twist of a $G_\QQ$-module $M$ for $j\in\ZZ$.
Fix a Galois-stable $\cO$-lattice $T_f$ inside $V_f$ and define an $\cO$-lattice $T_\of$ inside $V_\of$ to be
\[
T_\of:=\Hom_\cO(T_f,\cO)(1-k).
\]
For $g=f$ or $\of$, we write $A_g=V_g/T_g$.

We write $\Qcyc$ for the cyclotomic $\Zp$-extension of $\QQ$ and let $\Gamma$ denote the Galois group $\Gal\left(\Qcyc/\QQ\right)$.
Suppose $L_n$ is the unique finite subextension of $\Qcyc$ with Galois group over $\QQ$ of order $p^n$.
Suppose $\Qpn$ is the completion of $L_n$ at the unique place above $p$.
The Iwasawa algebra $\Lambda=\cO\lb\Gamma\rb $ is defined to be $\varprojlim \cO[\Gamma/\Gamma_n]$, where $\Gamma_n= \Gamma^{p^n}$ and the connecting maps are projections.
After fixing a topological generator $\gamma$ of $\Gamma$, there is an isomorphism of rings $\Lambda\cong\cO\lb X\rb $, by sending $\gamma$ to $X+1$.
Given a $\Lambda$-module $M$, denote its Pontryagin dual by $M^\vee := \Hom_{\cO}\left( M, K/\cO\right)$.

We recall the notion of \emph{fine Selmer groups} (which is denoted by $\Sel_0$ in the present article) defined by Coates and Sujatha in \cite{CoatesSujatha_fineSelmer} (see also \cite{Wut-JAG}).
This is a subgroup of the classical Selmer group obtained by imposing stronger vanishing conditions at primes above $p$ (the precise definition is reviewed in \S\ref{sec:notation} below).
A deep result of Kato shows that the fine Selmer group over $\Qcyc$ is always cotorsion as a $\Lambda$-module regardless of whether $f$ is ordinary at $p$ or not, a fact that is not true for classical Selmer groups.
Based on results of Kato in \cite{Kato} and Perrin-Riou \cite{perrinriou00b}, it has been shown in \cite[Proposition~6.3]{lei11compositio} that classical Selmer groups are not $\Lambda$-cotorsion when $p$ is a good non-ordinary prime. See also \cite[Theorem 2.6]{CoatesSujatha_book}.
Fine Selmer groups can also be used to formulate an Iwasawa Main Conjecture for modular forms without $p$-adic $L$-functions \cite[Conjecture~12.10]{Kato}; this fact that can also be exploited while working with Coleman families of modular forms where the signed Selmer groups are unavailable.
A formulation of an Iwasawa Main Conjecture without $p$-adic $L$-functions for universal deformations is given in \cite[Section 5]{Nakamura_K}.

In this article, we shall prove two control theorems for fine Selmer groups.
The first one concerns passing between the fine Selmer group of the $\varpi^e$-torsion of a cofinitely generated $\cO$-module and the $\varpi^e$-torsion of the fine Selmer group of the whole module.

\begin{lthm}[{Theorem~\ref{thm:1}}]
\label{CT_1}
Let $e\geq 1$ be fixed integers and $M$ a cofinitely generated $\cO$-module equipped with a continuous $G_S(\QQ)$-action.
Let $\cF$ be a subfield of $\Qcyc$.
Then the natural map
\[
r: \Sel_0\left( M[\varpi^e]/\cF\right) \longrightarrow \Sel_0\left( M/\cF\right)[\varpi^e]
\]
has finite kernel and cokernel with order bounded independently of $e$.
\end{lthm}

The second control theorem concerns passing between fine Selmer groups over $\Qcyc$ and $L_n$.
It is divided into two parts.
The first part studies the fine Selmer groups of $\varpi^e$-torsion of a general cofinitely generated $\cO$-module.
The second part studies the fine Selmer groups attached to modular forms.
Let $F\in \Lambda$ be an irreducible polynomial.
Let $g=f$ or $\of$, and $j\in\ZZ$.
We write $A_g(j)_{F^m}$ for the tensor product $A_g(j)\otimes_\cO\Lambda/F^m$, equipped with the diagonal Galois action.

Before discussing a crucial hypothesis, which we call (H-base$_F$), we introduce some notation.
Let $\iota$ denote the involution on $\Lambda$ sending a group-like element of $\Gamma$ to its inverse. For any $\Lambda$-module $M$,  we write $M^\iota$
for the $\Lambda$-module which coincides with $M$ as a $\Z_p$-module, with the action of
$\Gamma$ given by
\[
\gamma \cdot_{\iota} x = \gamma^{-1}x
\textrm{ for }\gamma \in \Gamma \textrm{ and }x\in M.
\]

\vspace{0.2cm}

\noindent\textbf{Hypothesis (H-base$_F$):} For all  $v| N$, the cohomology group $H^0\left( \QQ_v, A_g(j)\otimes_{\cO} \Lambda/ \mathcal{F}\right)$  is finite for $(g,j,\mathcal{F})=(f,i,F)$ and $(\of,k-i,F^\iota)$.
\vspace{0.25cm}

\noindent We can now state our second control theorem:
\begin{lthm}[{Theorem~\ref{thm:2}}]
\label{CT_2}
Let $M$ be a cofinitely generated $\cO$-module equipped with a continuous action of $G_S(\QQ)$.
Let $n\geq 0$ be an integer.
\begin{itemize}
 \item[(i)]
Let $e\ge 1$ be an integer.
Then, the kernel and cokernel of the restriction map
\[
r_n: \Sel_0\left(M[\varpi^e]/L_n\right) \longrightarrow \Sel_0\left(M[\varpi^e]/\QQ_{\cyc}\right)^{\Gamma_n}
\]
are finite and bounded independently of $n$.
\item[(ii)] Let $F$ be a fixed irreducible distinguished polynomial and $m\geq 1$ an integer.
Suppose that $M=A_f(i)_{F^m}$ or $A_{\of}(k-i)_{F^{\iota,m}}$ and that (H-base$_F$) holds.
Then, the kernel and cokernel of the restriction map
\[
r: \Sel_0\left(M/\QQ\right) \longrightarrow \Sel_0\left(M/\QQ_{\cyc}\right)^{\Gamma}
\]
are finite.
\end{itemize}
\end{lthm}

Our results are slightly more general than those proven by Rubin in \cite[Proposition 7.4.4]{rubin14} (see also \cite{wuthrich2004fine, lim2020control} for similar control theorems for fine Selmer groups of abelian varieties).
We utilize our control theorems to study the $\Lambda$-structure of $\Sel_0\left(A_g(j)/\Qcyc\right)^\vee$ under duality.
More precisely, we study links between $U:=\Sel_0\left(A_f(i)/\Qcyc\right)^\vee$ and $V:=\Sel_0\left(A_\of(k-i)/\Qcyc\right)^{\vee,\iota}$.
The twists we consider originate from the perfect pairing of $G_\QQ$-modules
\[
T_f(i)\times T_\of(k-i)\longrightarrow \cO(1).
\]

For any finitely generated $\Lambda$-torsion module $M$, we denote its $F^\infty$-torsion (resp. $\varpi^\infty$-torsion) part appearing in the pseudo-isomorphism between $M$ and cyclic $\Lambda$-modules by $M(F^\infty)$ (resp. $M(\varpi^\infty)$)
(see \S\ref{sec:notation}).
In this article, we provide necessary and sufficient conditions for the equalities $U(F^\infty)=V(F^\infty)$ and $U(\varpi^\infty)=V(\varpi^\infty)$ in terms of growth conditions of the following localization maps:
\begin{align*}
\theta_{F,m,e}:H^1\left(G_S(\QQ),A_f(i)\otimes_{\mathcal{O}}{\Lambda/F^m}[\varpi^e]\right) &\longrightarrow H^1\left(\Qp,A_f(i)\otimes_{\mathcal{O}}{\Lambda/F^m}[\varpi^e]\right),\\
\theta_{n,e}: H^1\left(G_S(L_n),A_f(i)[\varpi^e] \right) &\longrightarrow H^1\left(\Qpn,A_f(i)[\varpi^e]\right).
\end{align*}

%It is easy to see (cf. \S \ref{sec:notation}) that (H-base$_F$) is satisfied if $H^0\left( \QQ_{\cyc,v}, A_g(j)\right)$ is finite for both $(g,j)=(f,i)$ and $(\of,k-i)$.

Our control theorems allow us to prove:

\begin{lthm}
\label{main corollary}
Under the notation introduced above, we have:
\begin{enumerate}[(i)]
    \item \label{thm_A}
    Let $F\in \Lambda$ be an irreducible distinguished polynomial such that (H-base$_F$) holds.
    Then, $U(F^\infty)=V(F^\infty)$ if and only if
    \[\absolute{\Image(\theta_{F,m,e})} \sim_e q^{e \deg(F^m)} \text{ for all integers $m\ge1$.}\]

    \item \label{thm_B}
   We have $ U(\varpi^\infty)=V(\varpi^\infty)$  if and only if \[ \absolute{\Image\left(\theta_{n,e}\right)} \sim_n q^{ep^n} \textrm{ for all integers $e\ge1$.}\]
     In particular, $U$ and $V$ have equal $\mu$-invariants if  $\absolute{\Image\left(\theta_{n,e}\right)} \sim_n q^{ep^n}$ for all $e\geq 1$.
\end{enumerate}
\end{lthm}
Here, $a_e\sim_e b_e$ signify that $a_e$ and $b_e$ are positive integers such that $a_e/b_e$ and $b_e/a_e$ are bounded independently of $e$ (but the bounds may depend on $m$).
Likewise, $a_n\sim_n b_n$ signify that $a_n$ and $b_n$ are positive integers such that $a_n/b_n$ and $b_n/a_n$ are bounded independently of $n$ (but the bounds may depend on $e$).
Apart from our control theorems, the proof of Theorem~\ref{main corollary} relies on global duality and global Euler characteristic formulae that the growth conditions on the localization maps are equivalent to criteria established by Greenberg (see Proposition~\ref{prop:greenberg})

\begin{remark}
We remind the reader that \cite[Conjecture~A]{jhasuj}, which generalizes \cite[Conjecture~A]{CoatesSujatha_fineSelmer} (see also \cite[Conjecture~1.2]{aribam2014mu}), predicts that $\mu$-invariants of $\Sel_0\left(A_\of(k-i)/\Qcyc\right)^\vee$ and $\Sel_0\left(A_f(i)/\Qcyc\right)^\vee$ are always zero.
Theorem~\ref{main corollary}\eqref{thm_B} asserts that if the condition $\absolute{\Image\left(\theta_{n,e}\right)} \sim_n q^{ep^n}$  holds, then the $\mu$-invariant of $\Sel_0\left(A_\of(k-i)/\Qcyc\right)^\vee$ vanishes if and only if that of $\Sel_0\left(A_f(i)/\Qcyc\right)^\vee$ vanishes.
Conversely, if the $\mu$-invariants of these fine Selmer groups are zero, then the growth condition $\absolute{\Image\left(\theta_{n,e}\right)} \sim_n q^{ep^n}$  holds.
\end{remark}

%The strategy of proving Theorem~\ref{main corollary} is as follows:
%\begin{itemize}
%\item In \S\ref{section: control}, we prove control theorems satisfied by the fine Selmer groups of modular forms.

%\item We show via global duality and global Euler characteristic formulae that the growth conditions on the localization maps are equivalent to criteria established by Greenberg (see Proposition~\ref{prop:greenberg}) that allow us to compare the $F^\infty$-torsion and $\varpi^\infty$-torsion of finitely generated torsion $\Lambda$-modules.
%\end{itemize}
The following theorem gives sufficient conditions for ($\textrm{H-base}_F$) to hold.
\begin{lthm}[Theorem~\ref{prop:main}]
Let $p$ be an odd prime and $N$ a square-free integer coprime to $p$.
Let $f \in S_{k}(\Gamma_0(N),\omega)$ be a newform with nebentypus character $\omega$ of conductor $M$.
Let $0 \leq i \leq k$ be an integer.
For a rational prime $\ell$, let $m_\ell$ denote the order of $\ell$ in $(\ZZ/p\ZZ)^\times$.
Suppose that for each $\ell | N$ the following hold:
\begin{enumerate}[(i)]
\item $\ell \not\equiv 1 \mod p$,
\item if $\ell | M$, then $m_\ell$ does not divide $(1-k+i)$ or $(1-i)$, and
\item if $\ell | \frac{N}{M}$, then $\gcd(m_\ell, \phi(M))=1$ and $m_\ell$ does not divide $k$ or $(k-2)$.
\end{enumerate}
\noindent Then, for all primes $v$ of $\Qcyc$ that divide $N$, we have $H^0\left( \QQ_{\cyc,v}, A_g(j)\right)$ is finite for both $(g,j)=(f,i)$ and $(\of,k-i)$.
\end{lthm}
In \S\ref{appendixA3}, some explicit examples are computed satisfying the hypotheses of the above theorem.

For elliptic curves $E/\QQ$, recall from \cite[Problem~0.7]{KP} the following problem posed by Greenberg:
\begin{equation}\label{Gr}\tag{Gr}
\Char_\Lambda \Sel_0\left(E/\Qcyc\right)^\vee\stackrel{?}{=}\left(\prod_{{e_n\ge 1},{n\ge0}}\Phi_n^{e_n-1}\right).
\end{equation}
Here, for $n\ge1$,
\[\Phi_n=\frac{\left(1+X\right)^{p^n}-1}{\left(1+X\right)^{p^{n-1}}-1}\in\Lambda\] denotes the $p^n$-th cyclotomic polynomial in $1+X$ and
\[
e_n=\frac{\rank E\left(L_n\right)-\rank E\left(L_{n-1}\right)}{p^{n-1}\left(p-1\right)},
\]
where $L_n$ denotes the unique sub-extension of $\Qcyc$ such that $[L_n:\QQ]=p^n$.
When $n=0$, we define $\Phi_0=X$ and $e_0=\rank E\left(\QQ\right)$.
The right-hand side of \eqref{Gr} is invariant under $\iota$; this suggests that we might expect there is a pseudo-isomorphism between $\Sel_0\left(E/\Qcyc\right)$ and $\Sel_0\left(E/\Qcyc\right)^\iota$.
Recent results of Nakamura \cite{Nak17} on functional equations of Kato's Euler systems suggest that the expectation may be reasonable.
Therefore, one is tempted to conjecture that the growth of the image conditions given in Theorem~\ref{main corollary} hold true.

\section*{Acknowledgements}
The authors thank Antonio Cauchi, Meng Fai Lim and Sujatha Ramdorai for helpful discussions during the preparation of this article.
DK acknowledges the support of the PIMS postdoctoral fellowship.
AL is supported by the NSERC Discovery Grants Program RGPIN-2020-04259 and RGPAS-2020-00096.
JR acknowledges postdoctoral research support from the Tata Institute of Fundamental Research, Mumbai during the later stage of preparing this article.
Parts of this work were carried out during the thematic semester "Number Theory -- Cohomology in Arithmetic" at Centre de Recherches Mathématiques (CRM) in fall 2020.
DK, AL and JR thank the CRM for the hospitality and generous supports.
Finally, we also thank the anonymous referee for their helpful and constructive comments and suggestions, which have greatly helped improve the exposition of the article.

\section{Setup and notation}\label{sec:notation}
Fix an algebraic closure $\overline{\QQ}$ of $\QQ$.
Then an algebraic extension of $\QQ$ is a subfield of this fixed algebraic closure, $\overline{\QQ}$.
Throughout, $p$ denotes a fixed rational odd prime.
Fix a finite set $S$ of primes of $\QQ$ containing $p$, the primes dividing the level of the fixed modular form $f$, and the unique archimedean prime.
Denote by $\QQ_S$, the maximal algebraic extension of $\QQ$ unramified outside $S$.
For every (possibly infinite) extension $L$ of $\QQ$ contained in $\QQ_S$, write $G_S\left({L}\right) = \Gal\left(\QQ_S/{L}\right)$.
Write $S\left(L\right)$ for the set of primes of $L$ above $S$.
If $L$ is a finite extension of $\QQ$ and $w$ is a place of $L$, we write $L_w$ for its completion at $w$; when $L/\QQ$ is infinite, it is the union of completions of all finite sub-extensions of $L$.

\begin{defn}
Let $M$ be a $\Zp$-module equipped with a continuous $G_S(\QQ)$-action.
\begin{itemize}
 \item[(i)] For any finite extension $L/\QQ$ and $j\in\{1,2\}$, set
\[K^{j}_v\left(M/L\right) = \bigoplus_{w|v} H^j\left(L_w, M\right),\]
where the direct sum is taken over all primes $w$ of $L$ lying above $v$.
\item[(ii)]For an infinite algebraic extension $\mathcal{L}/\QQ$, we define $K_v^{j}\left(M/\mathcal{L}\right)$ by taking the inductive limit of $K_v^{j}\left(M /L\right)$ over all finite extensions $L/\QQ$ contained in $\mathcal{L}$.
\item[(iii)]Let $L$ be an algebraic extension of $\QQ$ that is contained inside $\QQ_S$, we define the \textbf{fine Selmer group} of $M$ over $L$ as
\[
\Sel_0(M/L):=\ker\left(H^1(G_S(\QQ),M)\longrightarrow\bigoplus_{v\in S(L)} K^1_v(M/L)\right).
\]
\end{itemize}
\end{defn}
\noindent Since all primes are finitely decomposed in the cyclotomic $\Zp$-extension, we will henceforth simplify notation and write $\bigoplus_{v\in S(\Qcyc)} H^1\left(\QQ_{\cyc,v},M\right)$ in place of $\bigoplus_{v\in S(\Qcyc)} K^1_v(M/L)$.

From now on, we fix a uniformizer $\varpi$ of $K$.
Recall from the introduction that $\Gamma$ denotes the Galois group $\Gal\left( \QQ_{\cyc}/\QQ\right)$ and $\Lambda = \cO\lb \Gamma \rb$.

\begin{defn}
Let $F\in\Lambda$.
Write $\Lambda/F$ for the quotient module $\Lambda/\langle F\rangle$, where $\langle F\rangle$ denotes the $\Lambda$-ideal generated by $F$.
Consider it as a $G_S(\QQ)$-module via the projection $G_S(\QQ)\longrightarrow\Gamma$ and $\Gamma$ acts on $\Lambda/ F$ via the multiplication by identifying the elements of $\Gamma$ with the group-like elements in $\Lambda$.
For a $G_S(\QQ)$-module $M$, define $M_{F}$ to be the tensor product $M\otimes_{\cO} \Lambda/F$ equipped with the diagonal action by $G_S(\QQ)$.
\end{defn}

Note that action of $G_S(\Qcyc)$ on $\Lambda/F$ is trivial.
\begin{defn}
Let $g=f$ or $\of$ and $j\in\ZZ$.
For any algebraic extension $L/\QQ$ contained in $\QQ_S$ and $F\in\Lambda$, the \textbf{$F$-twisted fine Selmer group} of $A_g(j)$ over $L$ is defined as $\Sel_0\left(A_g(j)_F/L\right)$.
Similarly, we define the \textbf{$F$-twisted $\varpi^e$-fine Selmer group} of $A_g(j)$ over $L$ to be $\Sel_0\left(A_g(j)_F[\varpi^e]/L\right)$.
\end{defn}

Let $M$ be a finitely generated $\Lambda$-module, then there exists a pseudo-isomorphism
\[
M\sim\Lambda^a\oplus\bigoplus_{i=1}^s \Lambda/\varpi^{\alpha_i}\oplus\bigoplus_{j=1}^t\bigoplus_{\ell=1}^{u_j}\Lambda/F_j^{\beta_{j,\ell}},
\]
where $a,s,t\ge0$, $\alpha_i,\beta_{j,\ell}\ge1$ are integers, $F_j$'s are mutually coprime irreducible distinguished polynomials in $\Lambda$.
Also, the pseudo-isomorphism is unique up to rearrangements of the direct summands.
We write
\[
M(\varpi^\infty)=\bigoplus_{i=1}^s \Lambda/\varpi^{\alpha_i}
\]
and define the $\mu$-invariant of $M$ to be $\sum_{i=1}^s \alpha_i$.
Given an irreducible distinguished polynomial $F\in\Lambda$, we write
\[M(F^\infty)=
\begin{cases}
\bigoplus_{\ell=1}^{u_j}\Lambda/F_j^{\beta_{j,\ell}}&\text{if $F=F_j$ for some $j$,}\\
0&\text{otherwise.}
\end{cases}
\]
The following hypothesis will be used whenever necessary throughout the article.

\noindent
\textbf{Hypothesis (H-base$_F$):} For all primes $v| N$, $H^0\left( \QQ_v, A_g(j)\otimes_{\cO} \Lambda/ \mathcal{F}\right)$ is finite for $(g,j,\mathcal{F})=(f,i,F)$ and $(\of,k-i,F^\iota)$, where $F$ is an irreducible distinguished polynomial in $\Lambda$.

The following stronger hypothesis, which we call (H-cyc), implies (H-base$_F$) for \emph{any} $F$.

\noindent
\textbf{Hypothesis (H-cyc):}
For all primes $v$ of $\Qcyc$ that divide $N$, we have $H^0\left( \QQ_{\cyc,v}, A_g(j)\right)$ is finite for both $(g,j)=(f,i)$ and $(\of,k-i)$.

\noindent Indeed, the above assertion follows from the observation that
\[
H^0\left(\QQ_v,A_g(j)\otimes_\cO\Lambda/ \cF\right)\subset
H^0\left(\QQ_{\cyc,v},A_g(j)\otimes_\cO\Lambda/\cF\right)=H^0\left(\QQ_{\cyc,v},A_g(j)\right)\otimes_\cO\Lambda/\cF.
\]
This stronger hypothesis (H-cyc) is verified for certain explicit examples in \S\ref{appendix}.

\section{Control theorems}
\label{section: control}
In this section, we prove two control theorems required for the proof of our theorems.
We fix an irreducible distinguished polynomial $F$ and an integer $m\ge1$.
%The first control theorem allows us to pass from a $F^m$-twisted $p$-primary fine Selmer group to a $F^m$-twisted $\varpi^e$-fine Selmer group by taking $\varpi^e$-torsions, whereas the second control theorem allows us to pass from a fine Selmer group over $\Qcyc$ to its counterpart over $\QQ$ by taking $\Gamma$-invariant.
\subsection{First control theorem}
\label{section: first control theorem}
Here, we prove a control theorem that allows us to go between the $\varpi^e$-fine Selmer groups of $A_g(j)$ and the $\varpi^e$-torsion subgroup of the fine Selmer group of $A_g(j)$.

Let $G$ be a profinite group and $M$ an $\cO$-module equipped with a continuous $G$-action.
The short exact sequence
\[
0 \longrightarrow M[\varpi^e] \longrightarrow M \xrightarrow{\varpi^e} M \longrightarrow 0
\]
induces the following short exact sequence in Galois cohomology for $j\geq 1$,
\begin{equation}\label{eq:SES-pe}
 0 \longrightarrow H^{j-1}\left( G, M\right)/\varpi^e \longrightarrow H^j\left( G, M[\varpi^e]\right) \longrightarrow H^j\left(G, M\right)[\varpi^e]\longrightarrow 0,
\end{equation}
which is a crucial ingredient in proving the first control theorem.
We also need the following lemma.
\begin{lemma}
\label{bound for local kernel}
Let $e\geq1$ be an integer.
Let $N$ be a cofinitely generated $\cO$-module.
Then $\absolute{N/\varpi^e}$ is bounded independently of $e$.
\end{lemma}

\begin{proof}
Recall that $K/\Qp$ is a finite extension with ring of integers $\cO$.
We first write
\[
N \simeq \left( K/\cO\right)^{d_0} \oplus N_{\rm{finite}} \quad \textrm{ where } d_0\geq 0.
\]
Since $K/\cO$ is divisible, it follows that $\left( K/\cO\right)^{d_0}/\varpi^e =0$.
Therefore, we only need to study the finite part.
Note that
\[
N_{\rm{finite}} \simeq \bigoplus_{i=1}^t \cO/\varpi^{n_i} \quad \textrm{ with }t\leq d,n_i\ge1.
\]
Therefore,
\[
\left(\cO/\varpi^{n_i}\right)/ \varpi^e \simeq
\begin{cases}
\cO/\varpi^{n_i} & \textrm{ if } e\geq n_i\\
\cO/\varpi^{e} & \textrm{ if } e< n_i.\\
\end{cases}
\]
In either case, we have that $\absolute{\left(\cO/\varpi^{n_i}\right)/\varpi^e} \leq q^{n_i}$.
The result follows.
\end{proof}

We now prove our first control theorem (see Theorem~\ref{CT_1}).
\begin{theorem}\label{thm:1}
Let $e\geq 1$ be fixed integers and $M$ a cofinitely generated $\cO$-module equipped with a continuous $G_S(\QQ)$-action.
Let $\cF$ be a subfield of $\Qcyc$.
Then the natural map
\[
r: \Sel_0\left( M[\varpi^e]/\cF\right) \longrightarrow \Sel_0\left( M/\cF\right)[\varpi^e]
\]
has finite kernel and cokernel with order bounded independently of $e$.
\end{theorem}

\begin{proof}
We have the following commutative diagram with exact rows
\[
\xymatrix{
0\ar[r]& \Sel_0\left(M[\varpi^e]/\cF\right)\ar[d]^{r}\ar[r]&H^1\left(G_S\left(\cF\right), M[\varpi^e]\right)\ar[d]^{h}\ar[r]& \displaystyle\bigoplus_{v\in S\left(\cF\right)} H^1\left(\cF_{v}, M[\varpi^e]\right)\ar[d]^{\gamma=\oplus \gamma_{v}}\\
 0\ar[r]&\Sel_0\left(M/\cF\right)[\varpi^e]\ar[r]&H^1\left(G_S\left(\cF\right), M\right)[\varpi^e]\ar[r]& \displaystyle \bigoplus_{v\in S\left(\cF\right)} H^1\left(\cF_{v}, M\right)[\varpi^e] }
\]
where $h$ (resp. $\gamma_{v}$) arise from the short exact sequence \eqref{eq:SES-pe} with $j=1$ and $G= G_S\left(\cF\right)$ (or $G=\Gal\left(\overline{\cF_{v}}/\cF_{v}\right)$).
These maps are surjective, and we have \[\ker h = M^{G_{\cF}}/\varpi^e \quad \textrm{ and } \quad \ker \gamma = \bigoplus_{v\in S(\cF)}M^{G_{\cF_{v}}}/\varpi^e.\]
Note that
\[
\absolute{H^0\left( G,M\right)/\varpi^e} \leq \absolute{H^0\left( G,M\right)/\left(H^0\left( G,M\right)\right)_{\rm div}},\]
which is finite and independent of $e$.
Furthermore, all primes of $\QQ$ are finitely decomposed in $\Qcyc$.
Therefore, both $\ker h$ and $\ker\gamma$ are finite and bounded independently of $e$.
The result follows from the snake lemma.
\end{proof}

\subsection{Second control theorem}
\label{section: second control theorem}
We now prove a control theorem which allows us to go between fine Selmer groups over $\Qcyc$ and $L_n$.
We begin by proving several preliminary lemmas.

\begin{lemma}
\label{lemma: Imai-like}
Let $M=A_f(i)_{F^m}$ or $A_\of(k-i)_{F^{\iota,m}}$.
The group $H^0\left(\Qp\left(\mu_{p^\infty}\right),M\right)$ is finite.
\end{lemma}
\begin{proof}
We only consider the case $M=A_f(i)_{F^m}$ since the other case can be proved similarly.
By definition, we have
\[
M=A_f\otimes_{\cO}(\Lambda/F^m)(i).
\]
As the Galois group $G_{\Qp(\mu_{p^\infty})}$ acts trivially on $(\Lambda/F^m)(i)$, we have
\[
H^0\left(\Qp\left(\mu_{p^\infty}\right),M\right)=H^0\left(\Qp\left(\mu_{p^\infty}\right),A_f\right)\otimes_{\cO}(\Lambda/F^m)(i).
\]
Therefore, it suffices to show that $H^0(\Qp(\mu_{p^\infty}),A_f)$ is finite.
By local Tate duality, we have
\[
H^0\left(\Qp\left(\mu_{p^\infty}\right),A_f\right)^\vee\cong H^2_{\Iw}\left(\Qp\left(\mu_{p^\infty}\right),T_\of(k)\right)=\varprojlim H^2\left(\Qp\left(\mu_{p^n}\right),T_\of(k)\right).
\]
Recall that we have assumed $p\nmid N$.
Thus, by \cite[Theorem~12.5(3)]{Kato} we know that the localization of $H^2_{\Iw}\left(\Qp\left(\mu_{p^\infty}\right),T_\of(k)\right)$ at a height-one prime that does not contain $p$ is zero.
In particular, as $\Lambda$-modules, we have a pseudo-isomorphism \[
H^0\left(\Qp\left(\mu_{p^\infty}\right),A_f\right)^\vee\sim \bigoplus_{i=1}^s\Lambda/\varpi^{\alpha_i}
\]
for some integers $s,\alpha_i\ge0$.
But $A_f\cong (K/\cO)^{\oplus2}$ as an $\cO$-module.
Thus, $H^0\left(\Qp\left(\mu_{p^\infty}\right),A_f\right)^\vee$ is an $\cO$-module of rank at most 2.
This forces $\alpha_i=0$ for all $i$.
In particular, $H^0\left(\Qp\left(\mu_{p^\infty}\right),A_f\right)$ is a pseudo-null $\Lambda$-module as required.
\end{proof}

\begin{remark}
Note that
\[H^0\left(\Qcyc,M\right)\subset H^0\left(\Qp\left(\mu_{p^\infty}\right),M\right).\]
Therefore, $H^0\left( \QQ_{\cyc}, M\right)$ is also finite.
\end{remark}

\begin{lemma}\label{lem:finiteness}
Hypothesis (H-base$_{F}$) implies that $H^0\left(\QQ_v,A_f(i)_{F^m}\right)$ and $H^0\left(\QQ_v,A_\of(k-i)_{F^{\iota,m}}\right)$ are finite for all $v|N$.
\end{lemma}
\begin{proof}
When $m=1$, this is clear.
Suppose that $m>1$, the short exact sequence
\[
0\longrightarrow F\Lambda/F^m\longrightarrow \Lambda/F^m\longrightarrow \Lambda /F\longrightarrow 0
\]
gives rise to the following short exact sequence
\[
0 \longrightarrow A_f(i)_{F^{m-1}} \xrightarrow{\times F} A_{f}(i)_{F^m} \longrightarrow A_{f}(i)_F \longrightarrow 0.
\]
The left most injectivity follows from the $\cO$-freeness of $\Lambda/F$.
From this, we obtain the long exact sequence
\[
0\longrightarrow H^0\left(\QQ_v,A_f(i)_{F^{m-1}}\right)\longrightarrow H^0\left(\QQ_v,A_f(i)_{F^m}\right)
\longrightarrow H^0\left(\QQ_v,A_f(i)_F\right)\longrightarrow \cdots.\]
Therefore, the lemma follows from induction.
\end{proof}

\begin{lemma}\label{lemma Lim-like}
Let $M$ be a cofinitely generated $\cO$-module equipped with a continuous action of $G_S(\QQ)$.
Let $\mathcal{F}$ be either $\QQ$ or $\QQ_\ell$, where $\ell$ is a prime number dividing $pN$.
Suppose that $\cF_\infty/\cF$ is the cyclotomic $\Zp$-extension of $\mathcal{F}$ and $\mathcal{F}_n$ is the intermediate subfield of $\mathcal{F}_\infty$ with $[\mathcal{F}_n: \mathcal{F}]=p^n$.
Write $G_n = \Gal\left( \cF_\infty/\cF_n\right)$ and $M(\cF_\infty)=H^0(\cF_\infty,M)$.
\begin{itemize}
 \item[(i)]Let $e\ge1$ be a fixed integer.
 Then $H^1\left( G_n, M[\varpi^e](\mathcal{F}_\infty)\right)$ is finite and bounded as $n$ varies.
\item[(ii)]Suppose that $A_f(i)_{F^m}$ or $A_{\of}(i)_{F^{\iota,m}}$ and that (H-base$_F$) holds.
Then $H^1\left( G_0, M(\mathcal{F}_\infty)\right)$ is finite.
\end{itemize}
\end{lemma}

\begin{proof}
With our notation, $\cF_0 = \cF$ and $G_0 = \Gal(\cF_\infty/\cF)$.\newline
(i) As $M[\varpi^e]$ is finite, $M[\varpi^e](\cF_\infty)$ is also  finite.
Since $G_n$ is pro-cyclic, we have \[H^1\left(G_n,M[\varpi^e]\left(\cF_\infty\right)\right)\cong H_0\left(G_n,M[\varpi^e](\cF_\infty)\right).\]
This is finite and bounded by the order of $M[\varpi^e](\cF_\infty)$, which is independent of $n$, as required.

\noindent
(ii) Consider the exact sequence
\[
0\longrightarrow H^0(G_0,M(\cF_\infty))\longrightarrow M(\cF_\infty)\stackrel{\gamma-1}{\longrightarrow}M(\cF_\infty)\longrightarrow H_0(G_0,M(\cF_\infty))\longrightarrow 0,
\]
where $\gamma$ is a topological generator of $G_0$.
Since $H^0\left(G_0,M\left(\cF_\infty\right)\right)=H^0\left(\cF,M\right)$ is finite by Lemmas~\ref{lemma: Imai-like} and \ref{lem:finiteness}, we have
\[
M\left(\cF_\infty\right)_{\mathrm{div}}\subset (\gamma-1)M\left(\cF_\infty\right).
\]
Thus, $H^1\left(G_0,M\left(\cF_\infty\right)\right)\cong H_0\left(G_0,M\left(\cF_\infty\right)\right)$ is bounded by $M(\cF_\infty)/M(\cF_\infty)_{\mathrm{div}}$, which is finite.
This concludes the proof.
\end{proof}
We can now prove our second control theorem (see Theorem~\ref{CT_2}).
\begin{theorem}\label{thm:2}Let $M$ be a cofinitely generated $\cO$-module equipped with a continuous action of $G_S(\QQ)$.
Let $n\geq 0$ be an integer.
\begin{itemize}
 \item[(i)]
Let $e\ge 1$ be an integer.
Then, the kernel and cokernel of the restriction map
\[
r_n: \Sel_0\left(M[\varpi^e]/L_n\right) \longrightarrow \Sel_0\left(M[\varpi^e]/\QQ_{\cyc}\right)^{\Gamma_n}
\]
are finite and bounded independently of $n$.
\item[(ii)] Let $F$ be a fixed irreducible distinguished polynomial and $m\geq 1$ an integer.
Suppose that $M = A_f(i)_{F^m}$ or $A_{\of}(k-i)_{F^{\iota,m}}$ and that (H-base$_F$) holds.
Then, the kernel and cokernel of the restriction map
\[
r: \Sel_0\left(M/\QQ\right) \longrightarrow \Sel_0\left(M/\QQ_{\cyc}\right)^{\Gamma}
\]
are finite.
\end{itemize}
\end{theorem}

\begin{proof}
Let $\cM=M$ or $M[\varpi^e]$.
Consider the diagram
\[
\xymatrix{
0\ar[r] &\Sel_0\left(\cM/L_n\right)\ar[r]\ar[d]^{r_n}&H^1\left( G_S\left(L_n\right), \cM\right)\ar[r]\ar[d]^{h_n}& \bigoplus_{v\in S(L_n)}H^1\left( L_{n,v}, \cM\right)\ar[d]^{\gamma_n}\\ 0\ar[r]&\Sel_0\left(\cM/\QQ_{\cyc}\right)^{\Gamma_n}\ar[r]&H^1\left( G_S\left(\QQ_{\cyc}\right), \cM\right)^{\Gamma_n}\ar[r]& \bigoplus_{v\in S(\Qcyc)}H^1\left( \QQ_{\cyc,v}, \cM\right)^{\Gamma_n}}
\]
%Note that the direct sum makes sense for the rightmost term of the bottom sequence as primes are finitely decomposed in the cyclotomic $\ZZ_p$-extension.
The vertical maps $h_n$ and $\gamma_n$ are the natural restriction maps.
These induce the left vertical map.
By inflation-restriction, first observe that both $h_n$ and $\gamma_n$ are surjective since $\Gamma$ has $p$-cohomological dimension 1.
By the inflation-restriction exact sequence, we know that
\begin{align*}
\ker h_n = H^1\left(\Gamma_n, \cM(\Qcyc)\right) \quad \textrm{and} \quad
\ker \gamma_n = \bigoplus_{v} H^1\left(\Gamma_{n,v}, \cM(\QQ_{\cyc,v})\right),
\end{align*}
where $\Gamma_{n,v}=\Gal(\QQ_{\cyc,v}/\QQ_{n,v})$.
Therefore, in the setting of (i), by Lemma~\ref{lemma Lim-like}(i) both $\ker h_n$ and $\ker \gamma_n$ are finite and bounded independently of $n$.
In the setting of (ii), by Lemma~\ref{lemma Lim-like}(ii) both $\ker h_0$ and $\ker \gamma_0$ are finite.
The result follows by an application of the snake lemma.
\end{proof}
\begin{remark}
Part (ii) of Theorem~\ref{CT_2} is utilized in the proof of Theorem~\ref{main corollary}\eqref{thm_A}, whereas part (i) is employed to prove Theorem~\ref{main corollary}\eqref{thm_B}.
Note that we do not require the twist by $\Lambda/F^m$ (or $\Lambda/F^{\iota,m}$) when we apply part (i) in the proof of Theorem~\ref{main corollary}\eqref{thm_B}.
\end{remark}

\section{Proof of Theorem~\ref{main corollary}}\label{S:compare-finite}

\subsection{~}\label{sec:strategy}
We review the criteria to establish a pseudo-isomorphism between two cofinitely generated $\Lambda$-modules developed by Greenberg.
%The following result explains Greenberg's strategy, which has been neatly summarized in \cite[Proposition~3.6]{Kim08}.

\begin{proposition}\label{prop:greenberg}
Let $U$ and $V$ be two finitely generated torsion $\Lambda$-modules.
\begin{itemize}
 \item[(1)] Let $F\in\Lambda$ be an irreducible distinguished polynomial.
 Then for all integers $m\ge0$, $\corank_{\Zp}\left(\left(U^\vee\right)_{F^m}\right)^\Gamma=\corank_{\Zp}\left(\left(V^\vee\right)_{F^m}\right)^\Gamma$ if and only if $U(F^\infty)=V(F^\infty)$.
 \item[(2)]For every integer $e\ge0$, the quotient $\absolute{(U/\varpi^e)_{\Gamma_n}}/\absolute{(V/\varpi^e)_{\Gamma_n}}$ is bounded as $n$ varies
 if and only if $U(\varpi^\infty)=V(\varpi^\infty)$.
\end{itemize}
\end{proposition}
\begin{proof}
This result is essentially due to Greenberg \cite{greenberg89}.
The proofs follow from the arguments given in Lemmas 2.1, 2.2 and Proposition 2.3 in \cite[\S 2.1]{AhmedLim19}.
\end{proof}

\noindent For our purposes, the proposition will be applied to $U=\Sel_0\left(A_f(i)/\Qcyc\right)^\vee$ and $V=\Sel_0\left(A_\of(k-i)/\Qcyc\right)^{\vee,\iota}$.

\subsection{~}
Let us first introduce some notation.
Let $e$ be a fixed positive integer.
Let $\cM$ denote either $A_{f}(i)_{F^m}[\varpi^e]$ or $A_{f}(i)[\varpi^e]$, both of which are finite Galois modules of $p$-power order.
The Cartier dual $\cM^\dagger :=\Hom(\cM,\mu_{p^\infty})$ is
$A_{\of}(k-i)_{F^{\iota, m}}[\varpi^e]$ (resp. $ A_{\of}(k-i)[\varpi^e]$).
We check this for $\cM= A_f(i)[\varpi^e]$ (checking the claim for the other module is similar):
\begin{align*}
\left(A_f(i)[\varpi^e]\right)^{\dagger} & = \left(A_f(i)\right)^{\dagger}/\varpi^e\\
& = T_{\of}(k-i)/\varpi^e\\
 & = A_{\of}(k-i)[\varpi^e].\\
\end{align*}
%This is also a $\Gal(\QQ_S/\QQ)$-module.
Let $\K$ be either $\QQ$ or $L_n$.
For $j=1,2$, we define the maps
\begin{align*}
\lambda_{\cM}^{(j)}&: H^j \left( G_S(\K), \cM\right) \longrightarrow \bigoplus_{v\in S} H^j\left(\K_{v}, \cM \right)\\
\lambda_{\cM^{\dagger}}^{(j)}&: H^j \left( G_S(\K), \cM^{\dagger}\right) \longrightarrow \bigoplus_{v\in S} H^j\left(\K_{v}, \cM^{\dagger} \right).
\end{align*}
Set
\begin{align*}
 K_j = K_j(\cM) = \ker\left( \lambda_{\cM}^{(j)}\right), &\qquad K_j^{\dagger} = K_j(\cM^{\dagger}) = \ker\left( \lambda_{\cM^{\dagger}}^{(j)}\right),\\
 G_j = G_j(\cM) = \Image\left( \lambda_{\cM}^{(j)}\right),&\qquad G_j^{\dagger} = G_j(\cM^{\dagger}) = \Image\left( \lambda_{\cM^{\dagger}}^{(j)}\right).
\end{align*}
We note that $K_1$ (resp. $K_1^\dagger$) is the fine Selmer group of $\cM$ (resp. $\cM^\dagger$) over $\K$.

\begin{lemma}\label{lem:quotient}
We have the equality
\[
\frac{\absolute{K_1^{\dagger}}}{\absolute{K_1}}= \frac{ \absolute{G_1} \cdot \chi_{\glob}(\K,\cM)}{\absolute{H^0\left(G_S(\K),\cM\right)}\cdot \absolute{G_2}}.\]
\end{lemma}
\begin{proof}
By global duality (see \cite[Theorem 3.1(a)]{Tat62}), we have
\[
\absolute{K_1^{\dagger}} = \absolute{K_2} = \frac{\absolute{H^2\left(G_S(\K),\cM\right)}}{\absolute{G_2}}.
\]
The global Euler characteristic formula states that
\[
\chi_{\glob}\left(\K,\mathcal{M}\right) =\prod_{i=0}^{2} \absolute{H^j\left( G_S(\K), \cM\right)}^{(-1)^{i}}.
\]
This allows us to deduce
\begin{align*}
\absolute{K_1^{\dagger}} &= \frac{\absolute{H^1\left(G_S(\K),\cM\right)} \cdot \chi_{\glob}(\cM)}{\absolute{H^0\left(G_S(\K),\cM\right)}\cdot \absolute{G_2}} \\
&= \frac{\absolute{K_1} \cdot \absolute{G_1} \cdot \chi_{\glob}(\cM)}{\absolute{H^0\left(G_S(\K),\cM\right)}\cdot \absolute{G_2}},
\end{align*}
as required.
\end{proof}

\subsection{Proof of Theorem~\ref{main corollary}\eqref{thm_A}}
Fix an integer $m\geq 1$ and take $\cM=A_f(i)_{F^m}[\varpi^e]$ and $\K=\QQ$.
We work under the hypotheses (H-base$_F$).

\begin{lemma}\label{lem:quotientA}
With the notation of Lemma~\ref{lem:quotient},
\[
\frac{\absolute{K_1^{\dagger}}}{\absolute{K_1}}\sim_e \absolute{\Image\left(\theta_{F,m,e}\right)} \cdot \chi_{\glob}\left(\QQ,\cM\right).
\]
\end{lemma}
\begin{proof}
By Lemma~\ref{lemma: Imai-like}, we know that $\absolute{H^0\left(G_S(\K),\cM\right)}\sim_e1$ since \[H^0\left(G_S(\QQ),\cM\right)\subset H^0\left(G_S(\Qcyc),A_f\right)\otimes_{\cO}\Lambda/F^m(i).\]
To bound $\absolute{G_2}$, it is enough to observe that \[G_2\subset\bigoplus_{v\in S} H^2\left(\QQ_v,\cM\right)\cong \bigoplus_{v\in S} H^0\left(\QQ_v,\cM^\dagger\right)^\vee,\]
where the isomorphism is the local Tate duality.
But
\[
H^0\left(\QQ_v,\cM^\dagger\right)\subset H^0\left(\QQ_v,A_\of(k-i)_{F^{\iota,m}}\right),
\]
which is finite for all $v$ by Lemmas~\ref{lemma: Imai-like} and \ref{lem:finiteness}.
Therefore, $\absolute{G_2}\sim_e1$.

It remains to compare $\absolute{G_1}$ with $\absolute{\Image\left(\theta_{F,m,e}\right)}$.
Here, $\theta_{F,m,e}$ is the composition
\[
H^1\left(G_S(\QQ),\cM\right)\stackrel{\lambda_\cM^{(1)}}{\longrightarrow}\bigoplus_{v\in S}H^1\left(\QQ_v,\cM\right)\longrightarrow H^1\left(\Qp,\cM\right),
\]
where the last arrow is given by the projection map.
Therefore, it is enough to bound $H^1\left(\QQ_v,\cM\right)$ for $v|N$.
Since $\cM$ is a $p$-group, the local Euler characteristic formula gives
\begin{align*}
\absolute{H^1\left(\QQ_v,\cM\right)}&=\absolute{H^0\left(\QQ_v,\cM\right)}\absolute{H^2\left(\QQ_v,\cM\right)}\\
&\le \absolute{H^0\left(\QQ_v,A_f(i)_{F^m}\right)}\absolute{H^0\left(\QQ_v,A_\of(k-i)_{F^{\iota,m}}\right)},
\end{align*}
which is bounded independently of $e$ by Lemma~\ref{lem:finiteness} under (H-base$_F$).
The result now follows.
\end{proof}

Recall from Proposition~\ref{prop:greenberg}(1) that we need to show
\[
\corank_{\Zp}\Sel_0\left(A_f(i)_{F^m}/\Qcyc\right)^{\Gamma}=\corank_{\Zp}\Sel_0\left(A_\of(k-i)_{F^{\iota,m}}/\Qcyc\right)^{\Gamma}
\]
is equivalent to $\Image(\theta_{F,m,e})\sim_e q^{e\deg(F^m)}$.
By the second control theorem (see Theorem~\ref{CT_2}(ii)) this is equivalent to show
\[
\corank_{\Zp}\Sel_0\left(A_f(i)_{F^m}/\QQ\right)=\corank_{\Zp}\Sel_0\left(A_\of(k-i)_{F^{\iota,m}}/\QQ\right)
\]
is equivalent to $\Image(\theta_{F,m,e})\sim_e q^{e\deg(F^m)}$.
It follows from the Structure Theorem of cofinitely generated $\cO$-modules that
\[
\absolute{\Sel_0\left(A_f(i)_{F^m}/\QQ\right)[\varpi^e]}\sim_e\absolute{\Sel_0\left(A_\of(k-i)_{F^{\iota,m}}/\QQ\right)[\varpi^e]}
\]
is equivalent to $\Image(\theta_{F,m,e})\sim_e q^{e\deg(F^m)}$.
On applying Theorem~\ref{CT_1} (with $n=0$ and $e$ varying), it is enough to show that
\begin{equation}\label{eq:required-A}
 \absolute{\Sel_0\left(A_f(i)_{F^m}[\varpi^e]/\QQ\right)}\sim_e\absolute{\Sel_0\left(A_\of(k-i)_{F^{\iota,m}}[\varpi^e]/\QQ\right)}
\end{equation}
is equivalent to $\Image(\theta_{F,m,e})\sim_e q^{e\deg(F^m)}$.
Since $\QQ$ is a totally real field, the global Euler characteristic formula tells us that
\[
\chi_{\glob}\left(\QQ,\cM\right)=\frac{1}{\absolute{\cM^-}}=\frac{1}{q^{e\deg( F^m)}},
\]
where $\cM^-$ denotes the $-1$-eigensubspace of the complex conjugation in $\cM$.
On combining this with Lemma~\ref{lem:quotientA}, we deduce \eqref{eq:required-A}, as required.
This concludes the proof of Theorem~\ref{main corollary}\eqref{thm_A}.
\qed

\subsection{Proof of Theorem~\ref{main corollary}\eqref{thm_B}}
Fix an integer $e\ge1$, set $\cM=A_f(i)[\varpi^e]$, and $\K=L_n$.
%We work under the hypotheses that for all integers $r\ge 1$, the polynomial $(1+X)^{p^r}-1$ is coprime to the $\Lambda$-characteristic ideal of $H^2_{\Iw}\left(G_S(\Qcyc), T_f(i)\right)$ and that $\Sel_0\left(A_\of(k-i)/\Qcyc\right)^\vee$ has zero $\mu$-invariant.
We have the following analogue of Lemma~\ref{lem:quotientA}.

\begin{lemma}\label{lem:quotientB}
With the notation of Lemma~\ref{lem:quotient},
\[
\frac{\absolute{K_1^{\dagger}}}{\absolute{K_1}}\sim_n \absolute{\Image(\theta_{n,e})} \cdot \chi_{\glob}(L_n,\cM).
\]
\end{lemma}
\begin{proof}
By Lemma~\ref{lemma: Imai-like}, we know that $\absolute{H^0\left(G_S(\K),\cM\right)}\sim_n1$ since \[H^0\left(G_S(L_n),\cM\right)\subset H^0\left(G_S(\Qcyc),A_f(i)\right).\]
To bound $\absolute{G_2}$, observe that \[G_2\subset\bigoplus_{v\in S(L_n)} H^2(L_{n,v},\cM)\cong \bigoplus_{v\in S(L_n)} H^0\left(L_{n,v},\cM^\dagger\right)^\vee\]
by local Tate duality.
But $H^0\left(L_{n,v},\cM^\dagger\right)$ has at most $q^{2e}$ elements and the number of primes above $S$ are bounded as $n$ varies.
Therefore, $\absolute{G_2}\sim_n1$.

It remains to compare $\absolute{G_1}$ with $\absolute{\Image\left(\theta_{n,e}\right)}$.
Note that $\theta_{n,e}$ is the composition
\[
H^1\left(G_S(F_n),\cM\right)\stackrel{\lambda_\cM^{(1)}}{\longrightarrow}\bigoplus_{v\in S(L_n)}H^1\left(L_{n,v},\cM\right)\longrightarrow H^1\left(L_{n,p},\cM\right),
\]
where the last arrow is given by the projection map.
Since $\cM$ is a $p$-group, for $v\nmid p$, the local Euler characteristic formula gives
\[
\absolute{H^1\left(L_{n,v},\cM\right)}=\absolute{H^0\left(L_{n,v},\cM\right)}\absolute{H^2\left(L_{n,v},\cM\right)}= \absolute{H^0\left(L_{n,v},\cM\right)}\absolute{H^0\left(L_{n,v},\cM^\dagger\right)},
\]
which is bounded independently of $q^{4e}$.
Therefore, the result follows.
\end{proof}

Recall from Proposition~\ref{prop:greenberg}(2) that we need to show
\[
\absolute{\left(\Sel_0\left(A_f(i)/\Qcyc\right)[\varpi^e]\right)^{\Gamma_n}}\sim_n\absolute{\left(\Sel_0\left(A_\of(k-i)/\Qcyc\right)[\varpi^e]\right)^{\Gamma_n}}
\]
if and only if $\absolute{\Image(\theta_{n,e})}\sim_n q^{ep^n}$.
In view of %the control theorems (
Theorem~\ref{CT_1} (with $\cF=\Qcyc$) and Theorem~\ref{CT_2}(i)), this amounts to showing
\begin{equation}\label{eq:required-B}
 \absolute{\Sel_0\left(A_f(i)[\varpi^e]/L_n\right)}\sim_n\absolute{\Sel_0\left(A_\of(k-i)[\varpi^e]/L_n\right)}
\end{equation}
if and only if $\absolute{\Image(\theta_{n,e})}\sim_n q^{ep^n}$.
The global Euler characteristic formula tells us that
\[
\chi_{\glob}\left(L_n,\cM\right)=\frac{1}{\absolute{\cM^-}}=\frac{1}{q^{ep^n}},
\]
since $L_n$ is a totally real field.
Therefore, from Lemma~\ref{lem:quotientB} we deduce that \eqref{eq:required-B} holds if and only if the required growth condition is satisfied.
This concludes the proof of Theorem~\ref{main corollary}\eqref{thm_B}.
\qed

\section{Computing invariants of local Galois representations}
\label{appendix}
We write
\[
\rho_f \colon G_\Q \rightarrow \mathrm{GL}_2(K)
\]
for the representation realized by $V_f$.
Recall that we are using the geometric normalization, so that the determinant of $\rho_f$ is given by $\epsilon_p^{1-k}\omega^{-1}$, where $\epsilon_p$ is the $p$-cyclotomic character, and $\omega$ is the nebentypus character of $f$ of conductor $M$.
We will write $\rho_{f,i} := \epsilon_p^i \otimes \rho_f$ for the $i$-th Tate twist of $\rho_f$.

For any rational prime $\ell$, write $\Q_\ell$ for the completion of $\Q$ at $\ell$, and write $\Q^\infty_\ell$ for the cyclotomic $\Z_p$-extension of $\Q_\ell$ (which we may identify with $\Q_{\cyc,v}$, where $v$ is a place of $\Qcyc$ lying above $\ell$).
We denote by $G_\ell$ and $G_\ell^\infty$ their respective absolute Galois groups.
In this section, we explain how to verify (in some simple cases) the following condition:

\vspace{1.0ex}
\noindent \textbf{(H-cyc')} For all primes $\ell | N$, we have $H^0\left( \Q_\ell^\infty, \rho_{g,j}\right)=0$ for both $(g,j)=(f,i)$ and $(\overline{f},k-i)$.

\vspace{1.0ex}

Note that (H-cyc') is equivalent to (H-cyc) (since a non-trivial subspace in the $K$-vector space $H^0(\QQ_\ell^\infty,\rho_{g,j})$ gives an infinite divisible group inside $H^0\left(\QQ_\ell^\infty,A_g(j)\right)$, and vice versa).
The purpose of this section is to show that (H-cyc) is not too strong a hypothesis.

When referring to this hypothesis, we will often just refer to the pair $(f,i)$, which then determines the other pair $(\overline{f},k-i)$.

Rather than give a complete study of (H-cyc'), we seek only to illustrate that it does hold quite often, and that in some situations it can be verified easily.
To this end, we make the additional simplifying assumptions that $f$ is a newform and that $N$ is square-free.
(These assumptions simplify the classification of the local Galois representations $\rho_{f,i}|_{G_\ell}$ via the local Langlands correspondence by eliminating the possibility of supercuspidal representations.)

Let $\phi$ denote the Euler totient function.
We aim to prove the following result.

\begin{theorem}\label{prop:main}
Let $p$ be an odd prime and $N$ be a square-free integer coprime to $p$.
Let $f \in S_{k}(\Gamma_0(N),\omega)$ be a newform with nebentypus character $\omega$ of conductor $M$.
Let $0 \leq i \leq k$ be an integer.
For a rational prime $\ell$, let $m_\ell$ denote the order of $\ell$ in $(\Z/p\Z)^\times$.
Suppose that for each $\ell | N$ the following hold:
%\begin{enumerate}
\item[(i)] $\ell \not\equiv 1 \mod p$,
\item[(ii)] if $\ell | M$, then $m_\ell$ does not divide $(1-k+i)$ or $(1-i)$, and

\item[(iii)] if $\ell | \frac{N}{M}$, then $\gcd(m_\ell, \phi(M))=1$ and $m_\ell$ does not divide $k$ or $(k-2)$.

\noindent Then (H-cyc') is satisfied.
\end{theorem}

\begin{proof}
Under these assumptions, the fact that $H^0(\Q_\ell^\infty,\rho_{f,i})=0$ and $H^0(\Q_\ell^\infty,\rho_{\bar{f},k-i})=0$ for each $\ell | N$ follows from Propositions \ref{prop:principal-series} and \ref{prop:special} below.
\end{proof}

\begin{remark} If $\ell$ has large order in $(\Z/p\Z)^\times$ and $p \gg 3Mk$, then all three conditions will be met (except for possibly when $k=2$ or $i=1$).
In fact, since the value of $m_\ell$ is related to the splitting behavior of $\ell$ in $\Q(\mu_p)$, the Chebotarev density theorem implies that a positive density of primes $p$ can be used.
\end{remark}

\begin{remark}
We emphasize that Theorem \ref{prop:main} only gives \textit{sufficient} conditions for (H-cyc') to hold.
We certainly do not expect all of the assumptions herein, especially the square-free level hypothesis, to be necessary.
In any case, examples which can be verified using Theorem \ref{prop:main} abound, and we write down a few concrete examples at the end.
It is easy to find many more using a computer algebra system such as SageMath \cite{sage}.
\end{remark}

\subsection{A note on the cyclotomic character over $\Q_\ell^\infty$}\label{sec:cyclotomic}

Fix an odd prime $p$ and a prime $\ell \neq p$.
We are interested in the absolute Galois group of the field $\Q_\ell^\infty$, which is the cyclotomic $\Z_p$-extension of $\Q_\ell$.
Let us make the following elementary observation.

\begin{lemma}
\label{lem:no-p-roots-of-unity}
If $\ell \not\equiv 1 \mod p$, then $\Q_\ell^\infty$ contains no $p$-power roots of unity.
\end{lemma}
\begin{proof}
Since $\ell \neq p$, the extension $\Q_\ell^\infty / \Q_\ell$ is unramified, so the corresponding residue field is $\textbf{F} = \bigcup_{n \geq 1} \F_{l^{p^n}}$.
Again since $\ell \neq p$, the $p$-power roots of unity in $\Q_\ell^\infty$ map isomorphically onto the $p$-power roots of unity in $\textbf{F}$ under the natural reduction map.
But for each $n$, the only roots of unity in $\F_{\ell^{p^n}}$ have order dividing $\ell^{p^n} - 1$.
By Fermat's little theorem, $\ell^{p^n} \equiv \ell \mod p$, so our assumption that $\ell \not\equiv 1 \mod p$ ensures that $p$ does not divide $\ell^{p^n} - 1$ for any $n$, which completes the proof.
\end{proof}

An immediate corollary is that $\epsilon_p|_{G_\ell^\infty}$ is nontrivial, but since $\epsilon_p$ does become trivial over $\Q_\ell(\mu_{p^\infty})$, we see that $\epsilon_p|_{G_\ell^\infty}$ has finite order.
More precisely, we have the following lemma.

\begin{lemma}\label{lem:cyclotomic-char-not-trivial} Let $m_\ell$ denote the order of $\ell$ in $(\Z/p\Z)^\times$.
The restriction $\epsilon_p|_{G_\ell^\infty}$ of the $p$-cyclotomic character to $\Q_\ell^\infty$ has order $m_\ell$.
\end{lemma}
\begin{proof}
From the construction of the cyclotomic $\Z_p$-extension $\Q_\ell^\infty$ and the fact that it is unramified over $\Q_\ell$, we see that
\[
[\Q_\ell(\mu_{p^\infty}) \colon \Q_\ell^\infty] = [\F_\ell (\mu_p) \colon \F_\ell]=m_\ell.
\]
Since $\epsilon_p$ generates $\mathrm{Gal}(\Q_\ell(\mu_{p^\infty}) / \Q_\ell^\infty)$, this completes the proof.
\end{proof}

In the sections that follow, we will write $I_\ell$ for the inertia group inside $G_\ell$.
For $F/\Q_\ell$ a local field with absolute Galois group $G_F$, we will often view a character $\chi \colon G_F \rightarrow K^\times$ as a character $\chi \colon F^\times \rightarrow K^\times$ via the dense injection $F^\times \hookrightarrow G_F$ of local class field theory.

\subsection{Local descriptions of Galois representations}\label{sec:local-descriptions}

To check (H-cyc'), it is necessary to understand the restriction of $\rho_f$ to the decomposition group $G_\ell \subset G_\Q$ for each prime $\ell | N$.
(Note that since $p \nmid N$, we always have $\ell \neq p$.) The local Langlands correspondence for $n=2$ allows us to determine these restrictions explicitly.
We will work entirely on the Galois side and make little mention of the automorphic story.
Nevertheless, it is the classification on the automorphic side which permits the explicit descriptions of our local Galois representations.
The standard reference for this material is Diamond-Im \cite[$\S$11]{DI}, and a nicely-written modern treatment can be found in Loeffler-Weinstein \cite{LW}, but note that both of these references use the arithmetic normalization (i.e. the Hodge-Tate weights of $V_f$ are $0$ and $k-1$).
The geometric normalization that we are using is also used in a paper of Weston \cite[$\S$5]{Weston}.

%We make two remarks before undertaking our computations. First, we will be writing our Galois representations in matrix form with entries given by Galois characters, but we will also view these as characters of $E^\times$ via class field theory. Second,
\subsubsection{A few notes on our strategy} We will use two facts to simplify our computations.
First, we will consider the base change of each local representation to the algebraic closure $\overline{K}$ or $\overline{K}_\mathfrak{p}$ for $\mathfrak{p} \mid p$ a prime of $K$ above $p$, since the existence of eigenvectors for $\rho_f$ with $K$-rational eigenvalues is invariant under base change.
(Base-change also preserves principal series and twist of Steinberg representations.)

Second, it suffices to study the semisimplification of each local representation, since
\begin{equation}\label{eq:semisimple}
\mathrm{dim}_K H^0(G_\ell^\infty , \rho_f) \leq \mathrm{dim}_K H^0(G_\ell^\infty, \rho_f^\mathrm{ss}).
\end{equation}

Recall that $f$ is of square-free level $N$, with nebentypus character $\omega$ of modulus $M$.
Our computations are divided into two cases, depending on whether $\ell | M$ or $\ell | \frac{N}{M}$.

\subsubsection{Case $\ell | M$}

Suppose $\ell | M$, so that the nebentypus character $\omega$ is ramified at $\ell$.
In this case, the local representation corresponds to a principal series representation $ \pi(\chi_1,\chi_2)$ associated to two continuous characters $\chi_i : G_\ell \to \overline{K}$, where $\chi_1$ is ramified and $\chi_2$ is unramified.
The semisimplification of the associated Galois representation then satisfies
\begin{equation}\label{case1}
\rho_{f}|^\mathrm{ss}_{G_\ell} \otimes \overline{K} \simeq \chi_1 \oplus \chi_2.
\end{equation}
%In particular, we have

%\begin{equation}\label{case1-2}
%\rho_{f,i}|^\mathrm{ss}_{G_\ell} %\otimes \bar{K} \simeq \epsilon_p^i %\chi_1 \oplus \epsilon_p^i \chi_2.
%\end{equation}

\begin{proposition}\label{prop:principal-series}
Let $f\in S_k(\Gamma_1(N),\omega)$ be a newform of square-free level, let $m_\ell$ denote the order of $\ell$ in $(\Z/p\Z)^\times$.
Suppose $\ell \not\equiv 1 \mod p$ and $\ell | M$.
If $m_\ell \nmid (1-k+i)$, then $H^0(\Q_\ell^\infty,\rho_{f,i})=0$.
\end{proposition}
\begin{proof}
By \eqref{eq:semisimple} and \eqref{case1}, it suffices to check whether $\epsilon_p^i \chi_1$ and $\epsilon_p^i \chi_2$ are both nontrivial over $\Q_\ell^\infty$.

Let us first consider the case $i=0$.
In light of \eqref{case1}, upon taking the determinant of $\rho_{f}$ we have the identity
\[
\chi_1 \chi_2 |_{G_\ell^\infty} = \epsilon^{1-k}\omega^{-1}|_{G_\ell^\infty}.
\]
Since $\chi_1$ is ramified and $\chi_2$ is unramified, we must have
\[
\chi_2 |_{G_\ell^\infty} = \epsilon_p^{1-k}|_{G_\ell^\infty},
\]
so by Lemma \ref{lem:cyclotomic-char-not-trivial}, this is nontrivial provided $m_\ell \nmid 1-k$.

Similarly, since both $\chi_1$ and $\omega$ are ramified, we must have
\[
\chi_1|_{I_\ell} = \omega^{-1}|_{I_\ell}
\]
is nontrivial.
Since $\Q_\ell^\infty / \Q_\ell$ is unramified, we have $\chi_1|_{G_\ell^\infty}$ is nontrivial as desired.

Now suppose we twist $\chi_1$ and $\chi_2$ by $\epsilon_p^i$.
Since $\chi_1$ is ramified and $\epsilon_p$ is unramified, their product is nontrivial no matter the value of $i$.
On the other hand,
\[
\epsilon_p^i \chi_2|_{G_\ell^\infty} = \epsilon_p^{1-k+i}|_{G_\ell^\infty},
\]
so this is nontrivial provided $m_\ell$ does not divide $1-k+i$.
\end{proof}

\subsubsection{Case $\ell | \frac{N}{M}$}
Recall that $\omega$ has conductor $M$.
Now suppose $\ell | \frac{N}{M}$, so that $\omega$ is unramified at $\ell$ and the local representation corresponds to a special (twist of Steinberg) representation.
This translates on the Galois side to
\begin{equation}\label{special}
\rho_f|_{G_\ell} \otimes \overline{K}_\p \simeq \left( \begin{matrix} \epsilon_p \chi & \ast \\ 0 & \chi \end{matrix} \right)
\end{equation}
where $\chi : G_\ell \to \overline{K}^\times$ is an unramified character.

We point out that the following result depends only on the weight of the modular form and not on the specific twist.%satisfying $\chi(\ell)=a_\ell(f)$. In particular, $\chi$ is nontrivial if $a_\ell(f) \neq 1$.

\begin{proposition}\label{prop:special}
Let $f\in S_k(\Gamma_1(N),\omega)$ be newform of square-free level.
Suppose $\ell \not\equiv 1 \mod p$, and let $m_\ell$ denote the order of $\ell$ in $(\Z/p\Z)^\times$.
If $\ell | \frac{N}{M}$, assume that $\gcd(m_\ell,\phi(M))=1$ and $m_\ell \nmid k(k-2)$.
Then $H^0(\Q_\ell^\infty,\rho_{f,i})=0$.
\end{proposition}
\begin{proof}
Just as in the proof of Proposition \ref{prop:principal-series}, it suffices to show that neither $\epsilon_p^{i+1}\chi$ or $\epsilon_p^i \chi$ is trivial when restricted to $G_\ell^\infty$.
Note that by Lemma \ref{lem:cyclotomic-char-not-trivial}, at most one of these characters can be trivial.
Assume that one of them is, and suppose $\epsilon_p^j \chi$ is nontrivial where $\{j,j'\}= \{i,i+1\}$.
Upon restricting to $G_\ell^\infty$, the determinant then gives us
\[
\epsilon_p^j \chi|_{G_\ell^\infty} = \epsilon_p^{1-k}\omega |_{G_\ell^\infty}.
\]
On the other hand, our assumption that $\epsilon_p^{j'} \chi|_{G_\ell^\infty}$ is trivial implies $\chi|_{G_\ell^\infty} = \epsilon_p^{-j'}|_{G_\ell^\infty}$, so combining these yields
\begin{align*}
1 &= \epsilon_p^{1-k-j+j'} \omega|_{G_\ell^\infty} \\
&=\epsilon_p^t \omega |_{G_\ell^\infty}
\end{align*}
for $t \in \{-k,2-k\}$.
Since $\gcd(m_\ell, \phi(M))=1$, this implies $\omega|_{G_\ell^\infty}=1$ and $\epsilon_p^t|_{G_\ell^\infty}=1$.
But this is only possible if $m_\ell | t$ by Lemma \ref{lem:cyclotomic-char-not-trivial}.
\end{proof}

\subsection{Examples}\label{appendixA3}

In this final section, we give a few explicit examples of the application of Theorem \ref{prop:main}.
These examples were all found using SageMath \cite{sage}.

\begin{example}
Let $f$ be the newform of weight $2$ and level $13$ whose Fourier coefficients live in $\Q(\sqrt{-3})$, and whose $q$-expansion begins
\[
f=q + (-1-\zeta_6)q^2 + (-2+2\zeta_6)q^3 + \zeta_6 q^4 + (1-2\zeta_6)q^5 + \cdots
\]
where $\zeta_6$ is a primitive $6$-th root of unity.
In particular, $f$ and $\overline{f}$ are distinct.
The nebentypus character of $f$ has conductor $13$, so condition (iii) of Proposition \ref{prop:main} is irrelevant.
Thus, $(f,i)$ and $(\bar{f},k-i)$ satisfy (H-cyc') if one chooses a prime $p$ such that
\begin{enumerate}
    \item $p \geq 5$, so that $13 \not\equiv 1 \mod p$,
    \item the order of $13$ in $(\Z/p\Z)^\times$ does not divide $(1-i)$
\end{enumerate}
For instance, this happens for $p=5$ and $i \not\equiv 1 \mod 4$, and for $p=7$ with $i\equiv0\mod2$.
\end{example}

\begin{example}
Let $f$ be the unique newform of weight $4$ and level $11$, with $q$-expansion beginning
\[
f=q +(1+\beta)q^2+(-1-4\beta)q^3+(-4+2\beta)q^4+(1+8\beta)q^5 + \cdots
\]
where $\beta = \sqrt{3}$.
This form has trivial nebentypus, so condition (ii) of Proposition \ref{prop:main} is irrelevant, and $\phi(M)=1$, so it is only necessary to pick $p>5$, $p\neq 11$ such that $m_{11} \nmid 4$, and then any twist $i$ can be chosen.

Since $11^4 = 14641$, one can quickly check using Sage that the only primes for which the hypotheses of Theorem \ref{prop:main} are not satisfied are $2,3,5,11$, and $61$.

\end{example}

\begin{example}
Let $f$ be the newform of weight $4$ and level $10$ with nontrivial nebentypus character $\omega$ of conductor $5$ with $q$-expansion begins
\[
f=q + 2 \zeta q^2 -2 \zeta q^3 - q^4 + (-5-10 \zeta )q^5 + \cdots,
\]
The Fourier coefficients live in $\Q(\sqrt{-1})$, and we have written $\zeta = \sqrt{-1}$ to avoid a conflict of notation.
Since $M=5$ and $\frac{N}{M}=2$, all three conditions of Proposition \ref{prop:main} must be checked.
One finds, for instance, that $(f,i)$ and $(\bar{f},k-i)$ satisfy (H-cyc') for $p=7$ whenever $6 \nmid (i-3)(1-i)$.
\end{example}
%%%%%%%%%%%%%%%%%%%%%%%%%%%%%%%%%%%%%%%%%%%%%%%%

\bibliographystyle{acm}
\bibliography{references}

\end{document}